\documentclass[aop,seceqn,MSNbibl,citesort,dvips]{arximspdf}

\doi{10.1214/10-AOP589}
\volume{39}
\issue{4}
\pubyear{2011}
\firstpage{1591}
\lastpage{1606}

\makeatletter

\newtheorem{theorem}{Theorem}[section]
\newtheorem{proposition}[theorem]{Proposition}
\newtheorem{corollary}[theorem]{Corollary}
\newtheorem{lemma}[theorem]{Lemma}

\newproclaim{remark}{Remark}
\newproclaim{remarks}{Remarks}

\newcommand{\Var}{\operatorname{Var}}
\newcommand{\Span}{\operatorname{span}}
\newcommand{\conv}{\operatorname{conv}}

\newcommand{\R}{\mathbb{R}}
\newcommand{\E}{\mathbb{E}}
\newcommand{\PP}{\mathbb{P}}
\newcommand{\one}{\mathbf{1}}
\newcommand{\NN}{\mathcal{N}}
\newcommand{\EE}{\mathcal{E}}

\newcommand{\const}{\operatorname{const}}

\makeatother

\begin{document}
\begin{frontmatter}

\title{Approximating the moments of marginals of high-dimensional
distributions\thanksref{T1}}
\runtitle{Approximating the moments of marginals}

\thankstext{T1}{Supported in part by NSF Grants DMS-09-18623 and DMS-10-01829.}

\begin{aug}
\author[A]{\fnms{Roman} \snm{Vershynin}\corref{}\ead[label=e1]{romanv@umich.edu}}
\runauthor{R. Vershynin}
\affiliation{University of Michigan}
\address[A]{Department of Mathematics\\
University of Michigan\\
Ann Arbor, Michigan 48109\\
USA\\
\printead{e1}} 
\end{aug}

\received{\smonth{12} \syear{2009}}
\revised{\smonth{6} \syear{2010}}

%
\begin{abstract}
For probability distributions on $\R^n$, we study the optimal sample
size $N = N(n,p)$ that suffices to uniformly approximate the $p$th
moments of all one-dimensional marginals. Under the assumption that the
marginals have bounded $4p$ moments, we obtain the optimal bound $N =
O(n^{p/2})$ for $p > 2$. This bound goes in the direction of bridging
the two recent results: a theorem of Guedon and Rudelson [\textit{Adv.
Math.} \textbf{208} (2007) 798--823] which has an extra logarithmic
factor in the sample size, and a result of Adamczak et al. [\textit{J.
Amer. Math. Soc.} \textbf{23} (2010) 535--561] which requires stronger
subexponential moment assumptions.
\end{abstract}

%
\begin{keyword}[class=AMS]
\kwd[Primary ]{62H12}
\kwd[; secondary ]{46B09}
\kwd{60B20}.
\end{keyword}
\begin{keyword}
\kwd{High-dimensional distributions}
\kwd{marginals}
\kwd{statistical estimation}
\kwd{heavy-tailed distributions}
\kwd{random matrices}.
\end{keyword}

\end{frontmatter}

\section{Introduction}

\subsection{The estimation problem}

We study the following problem:
how well can one approximate one-dimensional marginals of a distribution
on $\R^n$ by sampling?
Consider a random vector $X$ in $\R^n$, and suppose we would like to compute
the $p$th moments of the marginals $\langle X,x\rangle$ for all $x \in
\R^n$.
To this end, we sample $N$ independent copies $X_1,\ldots,X_N$ of $X$,
compute the empirical moment from that sample and we hope that
it gives a good approximation of the actual moment,
%
%
\begin{equation} \label{approx}
\sup_{x \in S^{n-1}} \Biggl| \frac{1}{N} \sum_{i=1}^N |\langle X_i,
x\rangle|^p - \E|\langle X, x\rangle|^p \Biggr| \le\varepsilon.
\end{equation}
Indeed, by the law of large numbers this quantity converges to zero as
$N \to\infty$.
To understand the quantitative nature of this convergence one would like
to estimate the optimal sample complexity $N = N(n,p,\varepsilon)$ for which
(\ref{approx}) holds with high probability. For $p=2$ this problem is
equivalent
to approximating the covariance matrix of $X$ by a sample covariance matrix,
and it was studied in \cite{KLS,B,R,MP,A,ALPT}.
For $p \ne2$, the problem was also studied in \cite{BLM,GM,GR,M,ALPT}.

A well-known \textit{lower} bound for the sample complexity is $N
\gtrsim
n$ for $1 \le p \le2$
and $N \gtrsim n^{p/2}$ for $p \ge2$.
Guedon and Rudelson \cite{GR} prove the upper bound
$N = O(n^{p/2} \log n)$ for $p \ge2$
under quite weak moment assumptions\setcounter{footnote}{1}\footnote
{The constant implicit
in the $O(\cdot)$ notation in the sample complexity $N$ depends
only on the constants implicit in the assumptions (\ref{minimal}); the
same convention applies to
other results.}
%
%
\begin{eqnarray} \label{minimal}
\|X\|_2 &=& O\bigl(\sqrt{n}\bigr) \qquad\mbox{a.s.},\nonumber\\[-8pt]\\[-8pt]
(\E|\langle X, x\rangle|^p)^{1/p} &=& O(1) \qquad\mbox{for all } x \in
S^{n-1}.\nonumber
\end{eqnarray}
The logarithmic term cannot, in general, be removed from the sample complexity;
this can be seen by considering a random vector $X$ uniformly
distributed in a set of $n$ orthogonal vectors
of Euclidean norm $\sqrt{n}$.
On the other hand, Adamczak et al. \cite{ALPT} recently managed to
remove the logarithmic term
for random vectors $X$ uniformly distributed in an isotropic convex
body $K$ in $\R^n$,
showing that for such distributions one has $N = O(n)$ for $1 \le p \le
2$ and $N = O(n^{p/2})$ for $p \ge2$.
Their result actually holds for all random vectors $X$ that satisfy the
\textit{sub-exponential} moment
assumptions
%
%
\begin{eqnarray} \label{sub-exponential}
\|X\|_2 &=& O\bigl(\sqrt{n}\bigr) \qquad\mbox{a.s.},\nonumber\\[-8pt]\\[-8pt]
(\E|\langle X, x\rangle|^q)^{1/q} &=& O(q) \qquad\mbox{for all $q \ge1$
and $x \in S^{n-1}$}.\nonumber
\end{eqnarray}
A program aiming at understanding general empirical processes with
sub-exponential
tails is put forward by Mendelson \cite{M,Me}.

\subsection{Distributions with finite moments: Main result}

At this moment there is no complete understanding of which
distributions on $\R^n$ require
logarithmic oversampling and which do not.
Clearly there is a gap between the minimal moment assumptions (\ref
{minimal}) of \cite{GR} and the
subexponential assumptions (\ref{sub-exponential}) of \cite{ALPT}.
The present note makes a step toward closing this gap.\looseness=-1

Distributions with tails heavier than exponential frequently arise in
statistics, economics,
engineering and other exact sciences like geophysics and environmental science.
Heavy-tailed distributions are frequently used to model data that
exhibit large fluctuations (see, e.g., \cite{MeS,KM,AFT} and the
references therein).
A very basic theoretical example of a heavy-tailed random vector in $\R
^n$ is
$X = (\xi_1,\ldots,\xi_n)$ where $\xi_j$ are independent random
variables with mean
zero, unit variance and power-law tails $\PP\{ |\xi_j| > t \} \sim
t^{-q}$ for some fixed exponent $q > 2$
(e.g., normalized Pareto distrbution to mention a specific example).
Such random vectors clearly\vadjust{\goodbreak} satisfy $\E\|X\|_2^2 = n$, thus $\|X\|_2 =
O(\sqrt{n})$ with high probability.
Moreover, the marginals have moments $(\E|\langle X, x\rangle
|^{q'})^{1/q'} = O(1)$ for all $q'<q$,
but the higher moments (for $q' > q$) are infinite.

We shall show that a version of the result of Adamczak et al. \cite{ALPT}
holds \textit{under finite moment assumptions} for $p \ne2$;
specifically, the logarithmic oversampling is not needed if we
replace $p$ by $4p$ in the minimal moment assumptions (\ref{minimal}).
We shall first consider independent random vectors $X_i$ in $\R^n$
that satisfy
%
%
\begin{equation} \label{Xi}
\|X_i\|_2 \le K \sqrt{n} \qquad\mbox{a.s.},\qquad
(\E|\langle X_i, x\rangle|^q)^{1/q} \le L \qquad\mbox{for all } x \in
S^{n-1}.\hspace*{-28pt}
\end{equation}

\begin{theorem}[(Approximation of marginals)] \label{sampling}
Let $p>2$, $\varepsilon> 0$ and $\delta> 0$.
Consider independent random vectors $X_i$ in $\R^n$ which satisfy
(\ref{Xi})
for $q=4p$.
Let $N \ge C n^{p/2}$ where $C$ is a suitably large quantity
that depends (polynomially) only on $K,L,p,\varepsilon, \delta$.
Then with probability at least $1-\delta$ one has
%
%
\begin{equation} \label{eq sampling}
\sup_{x \in S^{n-1}} \Biggl| \frac{1}{N} \sum_{i=1}^N |\langle X_i,
x\rangle|^p - \E|\langle X_i, x\rangle|^p \Biggr|
\le\varepsilon.
\end{equation}
\end{theorem}
\begin{remark*}
1. A more elaborate version of this result is Theorem~\ref{sampling
sharp} below.
One can get more information on the probability in question using general
concentration of measure results as is done in \cite{GR}. One can also
modify the argument
to deduce a version of this result ``with high probability'' in spirit
of \cite{ALPT}, that is,
with probability converging to $1$ (at polynomial rate) as $n \to
\infty$.

2. A standard modification of the argument (as in \cite{ALPT}) gives
an optimal result
also in the range $1 \le p < 2$.
Namely, if the random vectors satisfy (\ref{Xi}) for some $q \ge4p$,
$q>4$, then the conclusion
(\ref{eq sampling}) holds for $N \ge C_{K,L,p,q,\varepsilon,\delta} n$.

3. The method of the present note does not seem to work for $p=2$; this
important and more difficult case
is addressed in \cite{V} with an oversampling by a possibly parasitic
$(\log\log n)^{c_{p,q}}$ factor.
\end{remark*}

The argument of this paper also yields sharp bounds on the norms
of random operators $\ell_2 \to\ell_p$. The following result is a
version of a result
of \cite{ALPT}, Corollary~4.12, proved there under the stronger
sub-exponential moment
assumptions (\ref{sub-exponential}).
\begin{theorem}[(Norms or random matrices)] \label{norm p>2}
Let $p > 2$ and $\delta> 0$.
Consider independent random vectors $X_i$ in $\R^n$ which satisfy
(\ref{Xi}) for $q=4p$.
Then the $N \times n$ random matrix $A$ with rows $X_1,\ldots,X_N$ satisfies
with probability at least $1-\delta$ that
\[
\|A\|_{\ell_2 \to\ell_p} \le C (n^{1/2} + N^{1/p}),
\]
where $C$ depends (polynomially) only on $K,L,p,\delta$.\vadjust{\goodbreak}
\end{theorem}

\subsection{On the boundedness assumptions}

Let us take a closer look on our assumptions (\ref{Xi})
on the distribution. The boundedness assumption $\|X_i\|_2 = O(\sqrt
{n})$ a.s.
seems to be too strong---even the standard Gaussian distribution in
$\R^n$ does not satisfy it.
We will observe that, although this assumption cannot be formally dropped,
it can be removed by slightly modifying the estimation
process---discarding the the sample
vectors $X_i$ that do not satisfy it.\looseness=-1

First, it is easy to see that the boundedness assumption
$\|X_i\|_2 = O(\sqrt{n})$ a.s. cannot be dropped from our results.
To this end, one easily constructs a random vector whose
Euclidean norm has sufficiently heavy tails\footnote{For example, one
can achieve this
by considering a version of a ``multidimensional Pareto'' distribution
\cite{MeS}---the product of the standard Gaussian random vector in $\R
^n$ by an
independent scalar
random variable $\xi$ with a power-law tail.}
so that $\max_{i \le N} \|X_i\|_2 \gg\sqrt{n}$ with high probability
for $N \gg1$
and, in particular, for the stated number of samples $N \sim n^{p/2}$.
Then the approximation inequality (\ref{eq sampling}) will fail.
Indeed, once we choose $x$ in the direction of the vector $X_i$ with
the largest Euclidean norm,
we will have with high probability
that $|\langle X_i,x\rangle|^p = \|X_i\|_2^p \gg n^{p/2}$, which will
force the
average of the $N$ terms in (\ref{eq sampling}) to be much larger than
$n^{p/2}/N \sim\const$
while $\E|\langle X_i, x\rangle| ^p = O(1)$.

As a side note, the last observation also shows that the sample size $N
\sim n^{p/2}$
in Theorem~\ref{sampling} is optimal.

Let us also note that the weaker boundedness assumption
%
%
\begin{equation} \label{exp norm}
(\E\|X_i\|_2^q)^{1/q} \le L \sqrt{n}
\end{equation}
follows automatically from the second (moment) assumption in (\ref{Xi}).
To see this, we represent
$\|X_i\|_2^2 = \sum_{j=1}^n Z_j$ where $Z_j = |\langle X_i, e_j\rangle|^2$
and where $(e_j)$ is an orthonormal basis in $\R^n$.
Then Minkowski's inequality yields (\ref{exp norm})
\begin{eqnarray*}
(\E\|X_i\|_2^q)^{2/q}
&=& \Biggl[ \E\Biggl( \sum_{j=1}^n Z_j \Biggr)^{q/2} \Biggr]^{2/q}
\le\sum_{j=1}^n ( \E Z_j^{q/2} )^{2/q}\\
&=& \sum_{j=1}^n ( \E|\langle X_i, e_j\rangle|^q )^{2/q}
\le L^2 n.
\end{eqnarray*}

Although, as we noticed before, the strong boundedness assumptions
cannot be dropped formally,
they can be easily transferred into the estimation process. Instead of
using all sample points $X_i$
in the approximation inequality (\ref{eq sampling}), one can only use
those with moderate norms,
$\|X_i\|_2 = O(\sqrt{n})$. This will produce a similar approximation
result without any boundedness assumption.
Just the previous moment assumption will suffice:
%
%
\begin{equation} \label{Xi moment}
(\E|\langle X_i, x\rangle|^q)^{1/q} \le L \qquad\mbox{for all } x \in S^{n-1}.
\end{equation}

\begin{corollary}[(Approximation of marginals: no boundedness
assumption)] \label{sampling no bdd}
Let $p>2$, $\varepsilon> 0$, $\delta> 0$ and $K>0$.
Consider independent\vadjust{\goodbreak} random vectors $X_i$ in $\R^n$ which satisfy
(\ref{Xi moment})
for $q=4p$.
Let $N \ge C n^{p/2}$ where $C$ is a suitably large quantity
that depends (polynomially) only on $K,L,p,\varepsilon, \delta$.
Denote
\[
I := \bigl\{ i \le N\dvtx \|X_i\|_2 \le K \sqrt{n} \bigr\}.
\]
Then with probability at least $1-\delta$ one has
\[
\sup_{x \in S^{n-1}} \Biggl| \frac{1}{N} \sum_{i \in I} |\langle X_i,
x\rangle|^p - \E|\langle X_i, x\rangle|^p \Biggr|
\le\varepsilon+ K^{p-q} L^q.
\]
\end{corollary}
\begin{pf}
Consider the events $\EE_i = \{ \|X_i\|_2 \le K \sqrt{n} \}$.
The conclusion then follows by applying Theorem~\ref{sampling} to the
random vectors
$\bar{X}_i = X_i \one_{\EE_i}$, which clearly satisfy (\ref{Xi}).
Noting that $|\langle\bar{X}_i, x\rangle|^p = |\langle X_i, x\rangle
|^p \one_{\EE_i}$, we obtain this way that
%
%
\begin{equation} \label{approx bar}
\sup_{x \in S^{n-1}} \Biggl| \frac{1}{N} \sum_{i=1}^N |\langle X_i,
x\rangle|^p \one_{\EE_i} - \E|\langle\bar{X}_i, x\rangle|^p \Biggr|
\le\varepsilon.
\end{equation}
To complete the proof, it remains to estimate the error
%
%
\begin{eqnarray} \label{error}
\bigl| \E|\langle X_i, x\rangle|^p - \E|\langle\bar{X}_i, x\rangle
|^p \bigr|
&=& \E|\langle X_i, x\rangle|^p \one_{\EE_i^c}\nonumber\\[-8pt]\\[-8pt]
&\le&( \E|\langle X_i, x\rangle|^q )^{p/q} ( \PP(\EE
_i^c) )^{1-p/q},\nonumber
\end{eqnarray}
where we used H\"older's inequality.
To estimate the probability of $\EE_i^c$ we use (\ref{exp norm})
which follows from our moment
assumption (\ref{Xi moment}) as we noticed before. By Chebyshev's
inequality we obtain
\[
\PP(\EE_i^c) = \bigl\{ \|X_i\|_2 > K \sqrt{n} \bigr\} \le(L/K)^q.
\]
Using this and moment assumption (\ref{Xi moment}) we conclude that
the error (\ref{error}) is bounded by
$L^p (L/K)^{q(1-p/q)} = L^q K^{p-q}$.
Therefore in (\ref{approx bar}) we can replace $\E|\langle\bar
{X}_i, x\rangle|^p$ by $\E|\langle X_i, x\rangle|^p$
by increasing the error bound $\varepsilon$ by $L^q K^{p-q}$. This
completes the proof.
\end{pf}
\begin{remarks*}
1. Of course one can achieve the approximation error $2\varepsilon$ in
Corollary~\ref{sampling no bdd}
by choosing the threshold $K = K(L,\varepsilon)$ sufficiently large.

2. For some distributions one may be able to show that with high probability,
%
%
\begin{equation} \label{max small}
\max_{i \le N} \|X_i\|_2 \le K \sqrt{n}
\end{equation}
for some moderate value of $K$ [ideally $K=O(1)$] and for the desired
sample size~$N$.
In this case, with high probability all events $\EE_i$ in
Corollary~\ref{sampling no bdd}
hold simultaneously, and therefore they can be dropped from the approximation
inequality. One thus obtains the same bound as in Theorem \ref
{sampling} except
for the extra error term $K^{p-q} L^q$.

This situation occurs, for example, in the estimation result Adamczak
et al. \cite{ALPT} mentioned above.\vadjust{\goodbreak}
For the uniform distribution on an isotropic convex body, the
concentration theorem of Paouris \cite{P}
implies that
$\PP(\|X_i\|_2 \ge K \sqrt{n}) \le\exp(-\sqrt{n})$.
By union bound this implies that (\ref{max small}) holds
with probability $1 - N \cdot\exp(-\sqrt{n})$, which is almost $1$
for sample sizes $N$ growing linearly
or polynomially in $n$.
This is why in the final result of \cite{ALPT}
for uniform distributions on convex bodies no boundedness assumption is
needed, whereas for general
subexponential distributions one needs the boundedness assumption $\|
X_i\|_2 = O(\sqrt{n})$ a.s.\vspace*{-2pt}
\end{remarks*}

\subsection{\texorpdfstring{Heuristics of the proof of Theorem \protect\ref{sampling}}
{Heuristics of the proof of Theorem 1.1}}

Bourgain \cite{B} first demonstrated that proving deviation estimates
like (\ref{eq sampling})
reduces to bounding the contribution to the sum of the large
coefficients---those
for which $|\langle X_i, x\rangle| > B$ for a suitably large fixed
level $B$. Such reduction is used in
some of the later approaches to the problem \cite{GM,ALPT} as well as
in the present note.
However, after this reduction we use a different route.
Suppose for some vector $x \in S^{n-1}$ there are $s = s(B)$ large
coefficients as above.
The new ingredient of this note is a decoupling argument which is
formalized in Proposition~\ref{decoupling}.
It transports the vector $x$ into the linear span of at most $0.01 s$
of these $X_i$,
while approximately retaining the largeness of the coefficients,
$|\langle X_i, x\rangle| > B/4$.
Let us condition on these $0.01s$ random vectors $X_i$. On the one hand,
we have reduced the ``complexity'' of the problem---our $x$ now lies in
a fixed $0.01s$-dimensional subspace,
which has an $\frac{1}{2}$-net in the Euclidean metric of cardinality
$e^{0.02s}$.
On the other hand, the inequality $|\langle X_i, x\rangle| > B$ holds
for the remaining $0.99s$
vectors $X_i$ of which $x$ is independent; by (\ref{Xi}) and
Chebyshev's inequality
this happens with probability $(L/B)^{qs}$. Choosing the level $B$
suitably large so that $(L/B)^{qs} \ll e^{-0.02s}$
allows us to take the union bound over the net, and therefore to
control the contribution of the
large coefficients.\vspace*{-2pt}

\subsection{Organization of the paper}

In Section~\ref{s: decoupling} we develop the decoupling argument.
We use it to control the contribution of the large coefficients in
Section~\ref{s: norm}.
This is formalized in Theorem~\ref{norm} where we estimate the norm of
a random matrix $A$ with rows
$X_i$ in the operator norm $\ell_2 \to\ell_{2,\infty}$, and also in
Lemma~\ref{large coefficients}.
In Section~\ref{s: sampling}, we deduce in a standard way the main
results of this note---Theorem~\ref{sampling} on approximating the
moments of marginals
and Theorem~\ref{norm p>2} on the norms of random matrices $\ell_2
\to\ell_p$.

In what follows, $C$ and $c$ will stand for positive absolute constants
(suitably chosen);
quantities that depend only on the parameters in question such as $K,L,p,q$
will be denoted $C_{K,L,p,q}$.\vspace*{-2pt}

\section{Decoupling} \label{s: decoupling}\vspace*{-2pt}

\begin{proposition}[(Decoupling)] \label{decoupling}
Let $X_1,\ldots,X_s$ be vectors in $\R^n$ which satisfy the following
conditions
for some $K_1, K_2$:
%
%
\begin{equation} \label{Xi deterministic}
\|X_k\|_2 \le K_1 \sqrt{n},
\frac{1}{s} \sum_{i \le s, i \ne k} \langle X_i, X_k \rangle^2
\le K_2^4 n,\qquad
k = 1,\ldots,s.\vadjust{\goodbreak}
\end{equation}
Let $\delta\in(0,1)$ and let $B \ge C \delta^{-3/2} K_1$, $M \ge C
\delta^{-1/2} K_2^2/K_1$.
Assume that there exists $x \in S^{n-1}$ such that
\[
\langle X_i, x \rangle\ge B \sqrt{n/s} + M,\qquad i=1,\ldots,s.
\]
Then there exist a subset $I \subseteq\{1,\ldots,s\}$, $|I| \ge
(1-\delta)s$,
and a vector $y \in S^{n-1} \cap\Span(X_i)_{i \in I^c}$ such that
\[
\langle X_i, y \rangle\ge\tfrac{1}{4} \bigl(B \sqrt{n/s} + M\bigr),\qquad
i\in I.
\]
\end{proposition}
\begin{pf}
Without loss of generality, we may assume that $\delta>0$ is smaller
than a suitably chosen
absolute constant (this can be done by suitably increasing the value of
constant $C$).

\textit{Step} 1: \textit{Random selection}.
Denote
$a := B \sqrt{n/s} + M$.
Then
\[
\langle X_i/a, x\rangle\ge1,\qquad i=1,\ldots,s.
\]
The convex hull $K := \conv\{X_i/a, i=1,\ldots,s\}$ is separated
in $\R^n$ from the origin
by the hyperplane $\{ u\dvtx \langle u, x\rangle= 1 \}$.
By a separation argument, one can find a vector $\bar{x} \in\conv(K
\cup0)$, $\|\bar{x}\|_2 =1$
and such that
%
%
\begin{equation} \label{x bar}
\langle X_i/a, \bar{x}\rangle\ge1,\qquad i=1,\ldots,s.
\end{equation}
(Indeed, one chooses $\bar{x} = z/\|z\|_2$ where $z$ is the element of $K$
with the smallest Euclidean norm.)
We express $\bar{x}$ as a convex combination
\[
\bar{x} = \sum_{i=1}^s \lambda_i X_i/a\qquad
\mbox{for some }
\lambda_i \ge0,\qquad
\sum_{i=1}^s \lambda_i \le1.
\]
By Chebyshev's inequality, the set
$E := \{ i \le s \dvtx \lambda_i \le1/\delta s\}$
has cardinality $|E| \ge(1-\delta)s$.
We will perform a random selection on $E$.
Let $\delta_1,\ldots,\delta_s$ be i.i.d. selectors, that is,
independent $\{0,1\}$ valued
random variables with $\E\delta_i = \delta$. We define the random vector
\[
\bar{y} := \sum_{i \in E} \delta_i \lambda_i X_i/a
+ \sum_{i \in E^c} \delta\lambda_i X_i/a\qquad
\mbox{then } \E\bar{y} = \delta\bar{x}.
\]

\textit{Step} 2: \textit{Control of the norm and inner products.}
By independence and by definitions of $a$, $E$ and $B$ we have
\begin{eqnarray*}
\E\|\bar{y} - \delta\bar{x}\|_2^2
&=& \E\biggl\| \sum_{i \in E} (\delta_i-\delta) \lambda_i X_i/a
\biggr\|_2^2
= \sum_{i \in E} \E(\delta_i-\delta)^2 \cdot\lambda_i^2 \frac{\|
X_i\|_2^2}{a^2} \\
&\le& s \delta\cdot(1/\delta s)^2 \frac{K_1^2 n}{(B \sqrt{n/s})^2}
\le\frac{K_1^2}{\delta B^2}
\le0.1 \delta^2.
\end{eqnarray*}
By Chebyshev's inequality, we have with probability at least $0.9$ that
%
%
\begin{equation} \label{norm control}
\|\bar{y}\|_2 \le\|\bar{y} - \delta\bar{x}\|_2 + \|\delta\bar
{x}\|_2
\le2 \delta.
\end{equation}

Now fix $k \in E$. By definition of $\bar{y}$ and by (\ref{x bar}),
we have
%
%
\begin{equation} \label{exp y bar}
\E\langle X_k/a, \bar{y} \rangle
= \delta\langle X_k/a, \bar{x} \rangle
\ge\delta.
\end{equation}
We will need a similar bound with high probability rather than
in expectation. More accurately, we would like to bound below
\[
p_k := \PP\{ \langle(1-\delta_k) X_k/a, \bar{y} \rangle\ge
\delta/2 \}.
\]
Consider the random vector $\bar{y}^{(k)}$ obtained by removing from
the sum defining $\bar{y}$ the term corresponding to $X_k$
\[
\bar{y}^{(k)} := \sum_{i \in E, i \ne k} \delta_i \lambda_i X_i/a
+ \sum_{i \in E^c} \delta\lambda_i X_i/a
= \bar{y} - \delta_k \lambda_k X_k/a.
\]
Then $\bar{y}^{(k)}$ is independent of $\delta_k$, which gives
\[
p_k = \PP\{ \delta_k = 0 \} \cdot
\PP\bigl\{ \bigl\langle X_k/a, \bar{y}^{(k)} \bigr\rangle\ge\delta/2 \bigr\}.
\]
By definitions of $a$, $E$ and $B$ we can bound the contribution of the
removed term as
\[
\langle X_k/a, \lambda_k X_k/a \rangle
= \lambda_k \frac{\|X_k\|_2^2}{a^2}
\le(1/\delta s) \frac{K_1^2 n}{(B \sqrt{n/s})^2}
= \frac{K_1^2}{\delta B^2}
\le0.1 \delta^2.
\]
Then the random variable
$Z_k := \langle X_k/a, \bar{y}^{(k)} \rangle$
satisfies by (\ref{exp y bar}) that
\[
\E Z_k = \E\langle X_k/a, \bar{y} \rangle
- \E\langle X_k/a, \delta_k \lambda_k X_k/a \rangle
\ge\delta- 0.1 \delta^3
\ge0.9 \delta.
\]
Similar to the argument in the beginning of Step 2, we obtain
\begin{eqnarray*}
\Var Z_k
&=& \E(Z_k - \E Z_k)^2
= \E\biggl\langle X_k/a, \sum_{i \in E, i \ne k} (\delta_i-\delta)
\lambda_i X_i/a \biggr\rangle^2 \\
&=& \sum_{i \in E, i \ne k} \E(\delta_i-\delta)^2 \cdot\lambda
_i^2 \frac{\langle X_k, X_i\rangle^2}{a^4} \\
&\le&\delta\cdot\biggl( \frac{1}{\delta s} \biggr)^2 \frac{K_2^4 n
s}{(B \sqrt{n/s} + M)^4}
\le\frac{K_2^4}{\delta B^2 M^2}
\le0.01 \delta^3.
\end{eqnarray*}
By Chebyshev's inequality, we conclude that $\PP\{ Z_k \ge\delta/2
\} \ge1-\delta$.
We have shown that
\[
p_k \ge(1-\delta)(1-\delta) \ge1-2\delta.
\]

\textit{Step} 3: \textit{Decoupling.}
Denoting by $\EE_k$ the event
$\langle(1-\delta_k) X_k/a, \bar{y} \rangle\ge\delta/2$,
we have shown that $\PP(\EE_k) \ge1-2\delta$ for all $k \in E$.
Therefore with probability at least $0.9$,
at least $(1-20\delta)|E|$ of the events $\EE_k$ hold simultaneously.
Indeed, by linearity of expectation we have
\[
\E\sum_{k \in E} \one_{\EE_k^c}
= \sum_{k \in E} \PP(\EE_k^c)
\le2\delta|E|.
\]
By Chebyshev's inequality this yields
\[
\PP\biggl\{ \sum_{k \in E} \one_{\EE_k} \le(1-20\delta) |E| \biggr\}
= \PP\biggl\{ \sum_{k \in E} \one_{\EE_k^c} \ge20 \delta|E| \biggr\}
\le\frac{2 \delta|E|}{20 \delta|E|}
\le\frac{1}{10}.
\]
We have shown that with probability at least $0.9$ the following event occurs:
there exists a subset $I \subset E$, $|I| \ge(1-22\delta)s \ge
(1-22\delta)s$, such that $\EE_k$ holds
for all $k \in I$.

Assume the latter event occurs.
By definition of $\EE_k$ we clearly have $\delta_k=0$ whenever $\EE
_k$ holds.
Hence by definition of $\bar{y}$ one has $\bar{y} \in\Span(X_i)_{i
\in I^c}$.
Also, by definition of $\EE_k$, one has
\[
\langle X_k/a, \bar{y} \rangle\ge\delta/2,\qquad k \in I.
\]
Once we set $y := \bar{y}/\|\bar{y}\|_2$, this and (\ref{norm
control}) complete the proof.
\end{pf}

\section{\texorpdfstring{Norms of random operators $\ell_2 \to\ell_{2, \infty}$}
{Norms of random operators l2 to l2,infinity}}
\label{s: norm}

Recall that the weak $\ell_2$-norm $\|x\|_{2,\infty}$ of a vector $x
= (x_1,\ldots,x_N) \in\R^N$
is defined as the minimal number $M$ for which the nonincreasing
rearrangement $(x^*_k)$
of the sequence $(|x_k|)$ satisfies $x^*_k \le M k^{-1/2}$, $k=1,\ldots,N$.
It is well known that the quasi-norm \mbox{$\|\cdot\|_{2,\infty}$} is
equivalent to a norm on $\R^N$
(see \cite{SW}), and one can easily check that
$c_p \|x\|_p \le\|x\|_{2,\infty} \le\|x\|_2$ for all $p>2$.

Although \mbox{$\|\cdot\|_{2,\infty}$} is not a norm, for linear operators
$A \dvtx\R^n \to\R^N$
we will be interested in the ``norm'' $\|A\|_{\ell_2 \to\ell
_{2,\infty}}$ defined as the minimal
number $M$ such that $\|Ax\|_{2,\infty} \le M \|x\|_2$ for all $x \in
\R^n$.
\begin{theorem} \label{norm}
Consider independent random vectors $X_1,\ldots,X_N$ in $\R^n$
which satisfy (\ref{Xi}) for some $q > 4$.
Then, for every $t \ge1$, the random matrix $A$ whose rows are $X_i$
satisfies the
following with probability at least $1-Ct^{-0.9q}$.
For every index set $I \subseteq\{1,\ldots,N\}$,
one has
\[
\|P_I A\|_{\ell_2 \to\ell_{2,\infty}}
\le C_{K,L,q} \bigl[ \sqrt{n} + t \sqrt{|I|} ( N/|I|
)^{2/q} \bigr],
\]
where $P_I$ is the coordinate projection in $\R^N$ onto $\R^I$.
In particular, one has
\[
\|A\|_{\ell_2 \to\ell_{2,\infty}} \le C_{K,L,q} \bigl( \sqrt{n} + t
\sqrt{N} \bigr).
\]
\end{theorem}
\begin{remarks*}
1. This theorem is a finite-moment variant of Corollary 3.7 of~\cite
{ALPT}, where a similar result
is proved under the stronger sub-exponential moment assumptions (\ref
{sub-exponential}).
The latter is in turn a strengthening of an inequality of
Bourgain \cite{B} that has some
unnecessary logarithmic terms.

2. The conclusion of Theorem~\ref{norm} can be equivalently stated as follows.
For every subset $I \subseteq\{1,\ldots,N\}$, one has
\[
\biggl\| \sum_{i \in I} X_i \biggr\|_2
\le C_{K,L,q} \bigl[ \sqrt{n|I|} + t |I| ( N/|I| )^{2/q}
\bigr].
\]

3. It seems possible that Theorem~\ref{norm} holds for the spectral
norm $\|A\|_{\ell_2 \to\ell_2}$.
This would imply that Theorem~\ref{norm} holds in the important case $p=2$.
\end{remarks*}

The proof of Theorem~\ref{norm} is based on the decoupling
Proposition~\ref{decoupling}.
So we will first need to verify the assumptions on the vectors (\ref
{Xi deterministic}).
\begin{lemma} \label{decreasing rearrangement}
Let $Z_1, \ldots, Z_N \ge0$ be independent random variables which satisfy
$\E Z_i^q \le B^q$ for some $q > 0$ and some $B$.
Consider the nonincreasing rearrangement $(Z^*_i)$ of $(Z_i)$.
Then, for every $t \ge1$, one has with probability at least $1 - C
t^{-q}/N$ that
%
%
\begin{equation} \label{rearrangement}
Z^*_i \le tB (N/i)^{2/q},\qquad i=1,\ldots,N.
\end{equation}
In particular, for $q > 4$ (\ref{rearrangement}) implies
\[
\frac{1}{s} \sum_{i=1}^s (Z^*_i)^2 \le C_q t^2 B^2 (N/s)^{4/q},\qquad
s=1,\ldots,N.
\]
\end{lemma}
\begin{pf}
By homogeneity, we can assume that $B=1$. Then by Chebyshev's inequality
we have $\PP\{ Z_j > u \} \le u^{-q}$ for every $j \le N$ and $u>0$.
Now, if $Z^*_i > u$ then there exists a set $J \subseteq\{1,\ldots,N\}
$, $|J|=i$
such that $|Z_j| > u$ for all $j \in J$.
Taking union bound over possible choices of the
subsets $J$, using independence and Stirling's approximation, we obtain
for all $i=1,\ldots,N$
\[
\PP\{ Z^*_i > u \}
\le\pmatrix{N\cr i} \Bigl( \max_{j \le N} \PP\{ Z_j > u \} \Bigr)^i
\le\pmatrix{N\cr i} u^{-qi}
\le(e u^{-q} N / i)^i.
\]
Choosing $u = t (eN/i)^{2/q}$ we obtain
$\PP\{ Z^*_i > u \} \le(t^{-q} i / e N)^i$.
Then, for $t \ge1$,
\[
\PP\{ \exists i \le N \dvtx Z^*_i > u \}
\le\sum_{i=1}^N (t^{-q} i/eN)^i
\le C t^{-q}/N.
\]
This easily implies the first part of the lemma.
The second part follows by summation using that $\frac{1}{s} \sum
_{i=1}^s i^{-r} \le C_r s^{-r}$
for $0 < r < 1$; here $r = 4/q$.
\end{pf}
\begin{lemma} \label{Xi sum squared}
Consider independent random vectors $X_1,\ldots,X_N$ in $\R^n$
which satisfy (\ref{Xi}) for some $q > 4$.
Then for every $t \ge1$ the following holds with probability at least
$1 - C t^{-q}$.
For every subset $E \subseteq\{1,\ldots,N\}$ and every $k \le N$ one has
\[
\frac{1}{|E|} \sum_{i \in E, i \ne k} \langle X_i, X_k\rangle^2
\le C_q t^2 K^2 L^2 ( N/|E| )^{4/q} n.
\]
\end{lemma}
\begin{pf}
We fix $k \le N$ and apply Lemma~\ref{decreasing rearrangement} to the random
variables $Z^{(k)}_i := |\langle X_i, X_k\rangle|$, $i \le N$, $i \ne k$.
By assumptions (\ref{Xi}), we have $\E Z_i^q \le(KL\sqrt{n})^q$.
Then with probability at least $1 - C t^{-q}/N$, we have
\[
\frac{1}{s} \sum_{i=1}^s \bigl(\bigl(Z^{(k)}\bigr)^*_i\bigr)^2
\le C_q t^2 K^2 L^2 (N/s)^{4/q} n,\qquad s=1,\ldots,N.
\]
Taking union bound over $k \le N$ completes the proof.
\end{pf}
\begin{pf*}{Proof of Theorem~\ref{norm}}
By homogeneity, we can assume that $L=1$.
Also, by decomposing $I$ in three sets of roughly equal cardinality
we see that it suffices to prove the conclusion
for the subsets $I$ of cardinality $|I| \le N/2$.

Denote by $\EE$ the event in the conclusion of Lemma~\ref{Xi sum squared}.
If $\EE$ holds, then the assumptions (\ref{Xi deterministic}) of
decoupling Proposition~\ref{decoupling}
are satisfied for every $s$ and every subset $(X_i)_{i \in E}$,
$E \subseteq\{1,\ldots,N\}$, $|E|=s$, and with parameters
$K_1 = K$, $K_2^4 = C_q t^2 K^2 (N/s)^{4/q}$.
So, in view of application of decoupling Proposition~\ref{decoupling},
we consider
$B = B(K,\delta)$ and $M_1 = M_1(q,\delta, t)$ defined as
\[
B := C \delta^{-3/2} K_1,\qquad
M = C \delta^{-1} K_2^2/K_1
= C_q' \delta^{-1} t (N/s)^{2/q}
=: M_1 (N/s)^{2/q}.
\]
Note that we can assume that $C_q' \ge8$, which we will use later.

We will now need a convenient interpretation of the conclusion of the theorem.
Given $x \in S^{n-1}$, we denote by $|\langle X_{\pi(i)}, x\rangle|$
a nonincreasing rearrangement of the sequence $|\langle X_i, x\rangle
|$, $i=1,\ldots,N$.
Denote by $D$ the minimal number such that for every $x \in S^{n-1}$
and every $s \le N/2$ one has
\[
\bigl|\bigl\langle X_{\pi(s)}, x\bigr\rangle\bigr| \le R_s := D \bigl[ B \sqrt{n/s} +
M_1 (N/s)^{2/q} \bigr].
\]
Since $q \ge4$, the quantity $\sqrt{s}(N/s)^{2/q}$ is nondecreasing
in $s$.
Therefore one has for every $s \le m \le N/2$
\[
\bigl|\bigl\langle X_{\pi(s)}, x\bigr\rangle\bigr| \le D \bigl[ B \sqrt{n/s} + M_1 \sqrt
{m/s} (N/m)^{2/q} \bigr].
\]
It follows that for every $x \in S^{n-1}$, every $m \le N/2$,
and every index set $I \subseteq\{1,\ldots,N\}$, $|I| = m$, one has
\[
\|( \langle X_i, x\rangle)_{i \in I}\|_{2,\infty}
\le D \bigl[ B \sqrt{n} + M_1 \sqrt{m} (N/m)^{2/q} \bigr].
\]
If we are able to show that $D \le1$ with the high probability as
required in Theorem~\ref{norm},
this would clearly complete the proof.

Since the event $\EE$ holds with probability at least $1 - C t^{-q}$,
it suffices to show that the event
$\{ \EE\mbox{ and } D > 1\}$ occurs with probability at most $C
t^{-0.99 q}$.
Let us assume that the latter event does occur.
By definition of $D$, one can find an integer $s \le N/2$, a subset $E
\subseteq\{1,\ldots,N\}$, $|E|=s$
and a vector $x \in S^{n-1}$ such that
\[
|\langle X_i, x\rangle| \ge R_s, \qquad i \in E.
\]
By the definition of $R_s, B, M$ above, decoupling Proposition \ref
{decoupling} can be applied
for $(X_i)_{i \in E}$, and it yields the following. There exists a
decomposition $E = I \cup J$
into disjoint sets $I$ and $J$ such that $|I| \ge(1-\delta)s$, $|J|
\le\delta s$, and there exists a vector
$y \in\Span(X_j)_{j \in J}$, $\|y\|_2 = 1$, such that
%
%
\begin{equation} \label{Xi y}
|\langle X_i, y\rangle| \ge R_s/4,\qquad i \in I.
\end{equation}

Let $\beta= \beta(\delta) \ge0$ be a sufficiently small quantity to
be determined later.
Consider a $\beta$-net $\NN_J$ of the sphere $S^{n-1} \cap\Span
(X_j)_{j \in J}$.
As in known by volumetric argument (see, e.g., \cite{MS}, Lemma 2.6),
one can choose such a net
with cardinality
\[
|\NN_J| \le(3/\beta)^{|J|}.
\]
We can assume that the random set $\NN_J$ depends only on $\beta$ and
the random variables $(X_j)_{j \in J}$. There exists $y_0 \in\NN_J$
such that
$\|y-y_0\|_2 \le\beta$. By definition of $D$, this implies that
\[
\bigl|\bigl\langle X_{\pi(\lceil\delta s \rceil)}, y-y_0 \bigr\rangle\bigr|
\le R_{\lceil\delta s \rceil} \cdot\beta
\le R_{\delta s} \cdot\beta
\le\bigl(R_s / \sqrt{\delta}\bigr) \beta
= R_s/8,
\]
if we choose $\beta= \sqrt{\delta}/8$.
This means that all but at most $\delta s$ indices $i$ in $I$
satisfy the inequality $|\langle X_i, y-y_0 \rangle| \le R_s/8$, and therefore
[by (\ref{Xi y})] also the inequality $|\langle X_i, y_0 \rangle| \ge R_s/8$.
Let us denote the set of these coefficients by $I_0$.
Note that
\begin{eqnarray*}
R_s/8
&\ge&\frac{1}{8} M_1 (N/s)^{2/q}\qquad
\mbox{(by definition of $R_s$ and since $D > 1$)} \\
&\ge&\frac{C'_q}{8} (t/\delta) (N/s)^{2/q}\qquad
\mbox{(by definition of $M_1$)} \\
&\ge&(t/\delta) (N/s)^{2/q}\qquad
\mbox{(since $C'_q \ge8$)}.
\end{eqnarray*}

Summarizing, we have shown that the event $\{ \EE\mbox{ and } D > 1\}$
implies the following event that we call $\EE_0$:
there exist an integer $s \le N/2$, disjoint index subsets
$I_0 = I_0(s), J = J(s) \subseteq\{1,\ldots,N\}$ with cardinalities
$|I_0| \ge(1-2\delta)s$, $|J| \le\delta s$, and a vector $y_0 \in
\NN_J$ such that
\[
|\langle X_i, y_0\rangle| \ge(t/\delta) (N/s)^{2/q},\qquad i \in I_0.
\]
Note that by Chebyshev's inequality and independence, for a fixed $y_0
\in S^{n-1}$ and a
fixed set $I_0 \subset\{1,\ldots,N\}$ as above, one has
%
%
\begin{eqnarray} \label{probability for fixed}
\PP\{ |\langle X_i, y_0\rangle| \ge(t/\delta) (N/s)^{2/q},
i \in I_0 \}
&\le&\bigl( (t/\delta) (N/s)^{2/q} \bigr)^{-q|I_0|}
\nonumber\\[-8pt]\\[-8pt]
&=&\bigl( (\delta/t)^q (s/N)^2 \bigr)^{|I_0|}.\nonumber
\end{eqnarray}
Then we can bound the probability of $\EE_0$ by taking the union bound
over all $s, I_0, J$ as above, conditioning on the random variables
$(X_j)_{j \in J}$ (which fixes the net $\NN_J$), taking the union bound
over $y_0 \in\NN_J$, and finally evaluating the probability
using~(\ref{probability for fixed}). This yields
\[
\PP(\EE_0)
\le\sum_{s=1}^{N/2} \pmatrix{N\cr|I_0|} \pmatrix{N\cr|J|}
|\NN_J| \bigl( (\delta/t)^q (s/N)^2 \bigr)^{|I_0|}
\]
(recall that $I_0$ and $J$ in this sum may depend on $s$).
Also recall that with our choice $\beta= \sqrt{\delta}/8$, we have
$|\NN_J| \le(24/\sqrt{\delta})^{|J|}$.
Further, by our choice of $M_1$ we have $R_s/8 \ge\delta^{-1} t
(N/s)^{2/q}$.
Using Stirling's approximation, we obtain
\[
\PP(\EE_0)
\le\sum_{s=1}^{N/2} \biggl( \frac{eN}{|I_0|} \biggl( \frac{\delta
}{t} \biggr)^q \biggl( \frac{s}{N} \biggr)^2 \biggr)^{|I_0|}
\biggl( \frac{eN}{|J|} \cdot\frac{24}{\sqrt{\delta}} \biggr)^{|J|}.
\]
Estimating $s$ in the summand by $2|I_0|$ and using the inequalities
$|I_0| \ge(1-2\delta)s$ and $|J| \le\delta s$ along with monotonicity,
we conclude for a sufficiently small $\delta$ that
\[
\PP(\EE_0)
\le\sum_{s=1}^{N/2} \biggl( C \biggl( \frac{\delta}{t} \biggr)^q \frac
{s}{N} \biggr)^{(1-2\delta)s}
\biggl( \frac{CN}{\delta^{3/2} s} \biggr)^{\delta s}
\le\sum_{s=1}^{N/2} \biggl( \frac{t^{-q} s}{10N} \biggr)^{(1-3\delta)s}
\le t^{-0.9 q} N^{-0.9}.
\]
This completes the proof of Theorem~\ref{norm}.
\end{pf*}

\section{\texorpdfstring{Approximation of marginals and the $\ell_2 \to\ell_p$ norms of random operators}
{Approximation of marginals and the l2 to lp norms of random operators}} \label{s: sampling}

In this section we deduce from Theorem~\ref{norm} the main results of
this paper,
Theorems~\ref{sampling} and~\ref{norm p>2}.
The method of this deduction is by now standard; it was used in
particular in \cite{ALPT}.
It consists of an application of symmetrization, truncation, and
contraction principle,
and it reduces the problem to estimating the contribution to the sum of
large coefficients.

Specifically, given a threshold $B \ge0$ and a vector $x \in S^n$, we define
the set of large coefficients with respect to random vectors
$X_1,\ldots, X_N$ as
\[
E_B = E_B(x) = \{i \le N\dvtx |\langle X_i, x\rangle| \ge B\}.
\]
The truncation argument in the beginning of proof of Proposition 4.4 in
\cite{ALPT} yields
the following bound:
\begin{lemma}[(Reduction to the few large coefficients)] \label
{reduction to large}
Let $p \ge2$, $B \ge0$, $t \ge1$.
Consider independent random vectors $X_i$ in $\R^n$
which satisfy (\ref{Xi moment}) for $q=2p$.
Then for every positive integer $N$, with probability at least
\[
1 - \exp\bigl( -c \min\bigl(t^2 n B^{2p-2}, t \sqrt{Nn}/B\bigr) \bigr)
\]
one has
%
%
\begin{eqnarray} \label{three terms}
&&\sup_{x \in S^{n-1}} \Biggl| \frac{1}{N} \sum_{i=1}^N |\langle X_i,
x\rangle|^p - \E|\langle X_i, x\rangle|^p \Biggr|\nonumber\\
&&\qquad\le16 t B^{p-1} \sqrt{\frac{n}{N}}
+ \sup_{x \in S^{n-1}} \frac{1}{N} \sum_{i \in E_B(x)} |\langle X_i,
x\rangle|^p\\
&&\qquad\quad{} + \sup_{x \in S^{n-1}} \E\frac{1}{N} \sum_{i \in E_B(x)} |\langle
X_i, x\rangle|^p,
\nonumber
\end{eqnarray}
where $c=c_{p,K}>0$ depends only on $p$ and the parameter $K$
in the moment assumption (\ref{Xi moment}).
\end{lemma}

This lemma reduces the approximation problem in Theorem~\ref{sampling}
to finding an upper bound on the contribution of the large coefficients
$\frac{1}{N} \sum_{i \in E_B(x)} |\langle X_i,\break x\rangle|^p$.
In the following lemma, we observe that a slightly stronger bound
(for the \mbox{$\|\cdot\|_{2,\infty}$} norm rather than \mbox{$\|\cdot\|_p$} norm)
follows from Theorem~\ref{norm}.
To facilitate the notation, throughout the end of this section we will
write $a \lesssim b$ if
$a \le C_{K,L,p,q,\delta} b$.
\begin{lemma}[(Large coefficients)] \label{large coefficients}
Let $q > 4$, $t \ge1$, $\varepsilon\in(0,1)$ and $B \ge t
(\varepsilon N/n)^{2/(q-4)}$.
Consider independent random vectors $X_1,\ldots,X_N$ in $\R^n$
which satisfy (\ref{Xi}).
Then with probability at least $1 - C t^{-0.9q}$, one has for every $x
\in S^{n-1}$
\[
|E_B| \lesssim t^2 n / \varepsilon B^2,\qquad
\|(\langle X_i, x\rangle)_{i \in E_B}\|_{2,\infty} \lesssim t \sqrt
{n/\varepsilon}.
\]
\end{lemma}
\begin{pf}
By definition of the set $E_B$ and the norm \mbox{$\|\cdot\|_{2,\infty}$}
and using Theorem~\ref{norm},
we obtain with the required probability
%
%
\begin{equation} \label{BEB}
B^2 |E_B|
\le\|(\langle X_i, x\rangle)_{i \in E_B}\|_{2,\infty}^2
\lesssim n + t^2 |E_B| ( N/|E_B| )^{4/q}.
\end{equation}
It follows that $|E_B| \lesssim n/B^2 + N(t/B)^{q/2}$.
This and the assumption on $B$ implies that $|E_B| \lesssim t^2 n /
\varepsilon B^2$ as required.
Substituting this estimate into the second inequality in (\ref{BEB}),
we complete the proof.
\end{pf}
\begin{proposition}[(Deviation)] \label{sampling sharp}
Let $p>2$, $\varepsilon\in(0,1)$, $\delta>0$ and $N \ge
n/\varepsilon+ C$
where $C = C_{p,K,\delta}$ is suitably large.
Consider independent random vectors $X_i$ in $\R^n$
which satisfy (\ref{Xi}) for $q=4p$.
Then with probability at least $1 - \delta$ one has
%
%
\begin{equation} \label{eq sampling sharp}
\sup_{x \in S^{n-1}} \Biggl| \frac{1}{N} \sum_{i=1}^N |\langle X_i,
x\rangle|^p - \E|\langle X_i, x\rangle|^p \Biggr|
\lesssim\varepsilon^{1/2} + \frac{(n/\varepsilon)^{p/2}}{N} +
\biggl( \frac{n}{\varepsilon N} \biggr)^{3/2}.\hspace*{-28pt}
\end{equation}
\end{proposition}
\begin{remarks*}
1. Theorem~\ref{sampling} follows immediately from this result.

2. One could of course optimize the right-hand side in $\varepsilon$;
we did not do this in order
to make clear where the three terms come from.
\end{remarks*}
\begin{pf*}{Proof of Proposition~\ref{sampling sharp}}
We choose $B := t (\varepsilon N/n)^{2/(q-4)}$ so that Lemma \ref
{large coefficients} holds.\vadjust{\goodbreak}

Next, we choose $t = t(\delta,K)$ and $C = C_{p,K,\delta}$
sufficiently large so that the probabilities in
Lemmas~\ref{reduction to large} and~\ref{large coefficients}
are at least $1-\delta/2$ each. This is indeed possible for the probability
in Lemma~\ref{reduction to large}
as one can check that $t^2 n B^{2p-2} = t^{2p} \varepsilon N \ge
t^{2p}$ and
$t\sqrt{Nn}/B \ge N^{1/2 - 2/(q-4)} \ge C^{(p-2)/2(p-1)}$;
for the probability in Lemma~\ref{large coefficients} this is straightforward.

Let us assume that the conclusions of both these lemmas hold; as we now know
this holds with probability at least $1-\delta$. Our goal is to
estimate the three
terms in the right-hand side of (\ref{three terms}).

By our choice of $B$, the first term in the right-hand side of (\ref
{three terms}) is $\lesssim\varepsilon^{1/2}$
as required.
The second term can be bounded using Lemma~\ref{large coefficients}.
Since $\|\cdot\|_p \lesssim\|\cdot\|_{2,\infty}$ for $p>2$,
we obtain that
\[
\sup_{x \in S^{n-1}} \frac{1}{N} \sum_{i \in E_B} |\langle X_i,
x\rangle|^p
\lesssim\frac{1}{N} \|(\langle X_i, x\rangle)_{i \in E_B}\|
_{2,\infty}^p
\lesssim\frac{(n/\varepsilon)^{p/2}}{N}
\]
as required.
To compute the third term in the right-hand side of (\ref{three
terms}), consider for a fixed $x$ the random
variable $Z_i = |\langle X_i, x\rangle|$. Since $\E Z_i^q \le L^q$, we have
\[
\E Z_i^p \one_{\{Z_i \ge B\}}
\le\E Z_i^p (Z_i/B)^{q-p} \one_{\{Z_i \ge B\}}
\le\E Z_i^q / B^{q-p}
\le L^q B^{p-q}.
\]
Therefore, by our choice of $B$, we have
\begin{eqnarray*}
\sup_{x \in S^{n-1}} \E\frac{1}{N} \sum_{i \in E_B} |\langle X_i,
x\rangle|^p
&=& \sup_{x \in S^{n-1}} \frac{1}{N} \sum_{i=1}^N \E Z_i^p \one_{\{Z
\ge B\}} \\
&\le& L^q B^{p-q}
\lesssim\biggl( \frac{n}{\varepsilon N} \biggr)^{{3p}/({2(p-1)})}\\
&\le&\biggl( \frac{n}{\varepsilon N} \biggr)^{3/2}.
\end{eqnarray*}
Combining these estimates, we complete the proof.
\end{pf*}
\begin{remark*}
Theorem~\ref{norm p>2} now follows easily.
We can assume that $N \ge C$ where $C = C_{p,K,\delta}$ is suitably large.
Now, for $N \le n$ this result follows from Theorem~\ref{norm} since
$\|A\|_{\ell_2 \to\ell_p} \lesssim\|A\|_{\ell_2 \to\ell
_{2,\infty}}$.
For $N \ge n$, the result follows from Proposition~\ref{sampling sharp}
with $\varepsilon= 1$, noting that $(\E|\langle X_i, x\rangle
|^p)^{1/p} \le L$ as $p \le q$.
\end{remark*}

\section*{Acknowledgment} The author is grateful to the referees for
their thorough
reading of the first two versions of this manuscript and for many suggestions,
which greatly improved the presentation of this paper.

%

%
\printaddresses

\end{document}